\documentclass[10pt]{article}
\usepackage{amssymb}
\usepackage[dvipsnames]{xcolor}

\usepackage{amsmath, amssymb, amsthm, mathtools}
\usepackage{tikz}
\usetikzlibrary{arrows.meta}
\usetikzlibrary{positioning}
\usepackage{caption, subcaption, float}
\usepackage[makeroom]{cancel}
\usepackage[all, knot]{xy}
\usepackage{graphicx}
\usepackage{xcolor}
\usepackage{enumitem}
\DeclareSymbolFont{symbolsC}{U}{txsyc}{m}{n}
\DeclareMathSymbol{\strictif}{\mathrel}{symbolsC}{74}

\usepackage{todonotes}

\newtheorem{theorem}{Theorem}[section]

\newtheorem*{theorem*}{Theorem} 
\newtheorem*{lemma*}{Lemma}

\theoremstyle{definition}
\newtheorem{definition}[theorem]{Definition}

\newtheorem{proposition}[theorem]{Proposition}

\newtheorem{corollary}[theorem]{Corollary}

\theoremstyle{remark}

\theoremstyle{remark}

\newcommand\restr[2]{{
\left.\kern-\nulldelimiterspace 
#1 
\vphantom{\big|} 
\right|_{#2} 
}}

\newcommand{\A}{\mathfrak{A}}

\newcommand{\M}{\mathfrak{M}}

\let\phi\varphi
\title{A note on ultrahomogeneous unary algebras}
\author{Cheng Liao}
\date{}
\begin{document}

\maketitle

\begin{abstract}
In a recent paper \cite{q},  Quinn-Gregson fully classified ultrahomogeneous mono-unary algebras of arbitrary cardinalities. In particular, he proved that every 1-ultrahomogeneous mono-unary algebra is ultrahomogeneous, and every locally finite mono-unary algebra with finitely many 1-orbits is $\omega$-categorical.  It is also known that every 1-homomorphism
homogeneous mono-unary algebra is homomorphism-homogeneous \cite{ed}. Very naturally, Quinn-Gregson asked if these three results can be generalized to unary algebras. This short note answers the first and the third question negatively by providing two simple counterexamples. We also show that the second question has a positive answer in ``tree-like" cases which covers his result to mono-unary algebras and has a small combinatorial implication.    
\end{abstract}

$$$$

A \textit{mono-unary} algebra is a set with a single unary function. Because of their relatively simple structures, it allows Quinn-Gregson to fully classify ultrahomogeneous mono-unary algebras and ultrahomogeneous $\omega$-categorical mono-unary algebras in his recent paper \cite{q}. In particular, he proved the following results for mono-unary algebras:

\begin{theorem*}[A]\cite[Prop 4.1]{q}
    A mono-unary algebra is ultrahomogeneous if it is 1-ultrahomogeneous.
\end{theorem*}

\begin{theorem*}[B]\cite[Thm 5.9]{q}
    A countably infinite mono-unary algebra is $\omega$-categorical if and only if it is locally finite and has finitely many 1-orbits.
\end{theorem*}

Quinn-Gregson also noticed that it follows from the proofs of \cite[Lem 2 and 4]{ed} that every 1-homomorphism
homogeneous mono-unary algebra is homomorphism-homogeneous. He then asked if these three results also hold for unary algebras -- sets with finitely many unary functions. This note answers the first question (i.e., whether Theorem (A) holds for unary algebras) and the third question (i.e., whether 1-homomorphism
homogeneous unary algebras are homomorphism-homogeneous) negatively by giving two simple counterexamples. We also prove the corresponding Theorem (B) for unary algebras\footnote{We refer to this as the second question in the following. We will also use the terms ``the first question" and ``the third question" consistently in this note. } in a very special case which actually implies Quinn-Gregson's theorem and has a small combinatorial implication. To keep this note short, we refer to the introduction of \cite{q} for background. One important historical fact worth mentioning here is that there is substantial overlap between the above two theorems of Quinn-Gregson and the results obtained by a group of Russian mathematicians in the last century. In particular, Theorem (B) follows from Shishmarev's result obtained in \cite[Thm 1]{S} and a complete description of theories of mono-unary algebras with quantifier elimination was given by Ivanov in \cite[Thm 2.6]{I}\footnote{I would like to thank Alexander Ivanov for bringing this fact into my attention. For more details, see \cite{q}.}.

\section{Basics}

We record some basic definitions and facts. A \textit{partial unary} algebra $\A=(A,f_1,...,f_n)$ is a set $A$ with finitely many unary partial functions $f_i$'s. For any structure $\mathcal{M}$,  $Aut(\mathcal{M})$ is the set of all automorphisms of $\mathcal{M}$ which acts on $M^n$ for any $n\in \mathbb{N}^+$ in the obvious way. We call $Aut(\mathcal{M})$ \textit{oligomorphic} if the number of $n$-orbits (i.e., orbits of the action on $M^n$) is finite for any $n\in \mathbb{N}^+$. 

For a countably infinite structure $\mathcal{M}$, it is $\omega$-categorical if the first-order theory of $\mathcal{M}$ has a unique countably infinite model up-to-isomorphism. The following is the well-known Ryll-Nardzewski theorem, the proof of which can be found in any standard textbook about model theory (see for example \cite{2}):

\begin{theorem}
    A countably infinite structure $\mathcal{M}$ is $\omega$-categorical if and only if $Aut(\mathcal{M})$ is oligomorphic.
\end{theorem}

A structure $\M$ is \textit{locally finite} if for any $n\in\mathbb{N}$, every $n$-generated substructure of $\M$ is finite. It is \textit{uniformly locally finite} if there is a function $f: \mathbb{N}\rightarrow \mathbb{N}$ such that every $n$-generated substructure has no more than $f(n)$ many elements.

\begin{definition}
    A structure $\mathcal{M}$ is \textit{n-ultrahomogeneous} if every isomorphism between $n$-generated substructures of $\mathcal{M}$ extends to an automorphism of $\mathcal{M}$. $\mathcal{M}$ is \textit{ultrahomogeneous} if it is $n$-ultrahomogeneous for each $n\in \mathbb{N}$.    
    
    A structure $\mathcal{M}$ is \textit{n-homomorphism-homogeneous} if every homomorphism between $n$-generated substructures of $\mathcal{M}$ extends to an endomorphism of $\mathcal{M}$, and $\mathcal{M}$ is homomorphism-homogeneous if it is $n$-homomorphism-homogeneous for any $n\in \mathbb{N}$.
\end{definition}

\noindent\textbf{Acknowledgments}.      The author is very grateful to his supervisor Nick Ramsey for helpful comments which improved the first draft. The author is also thankful to Alexander Ivanov for his comments about the history of model theory of unary algebras. Finally, the author also want to thank Thomas Quinn-Gregson for making the counterexample in Proposition 2.3 slightly simpler than the original one.

\section{Counterexamples}

In this section, we give examples to show that the first and the third question have negative answers. For the first question, we define a unary algebra $\A=(A,f_1,f_2)$ with two unary functions $f_1,f_2$ as follows (if no $f_i$-arrow comes out of a point, that means $f_i$ maps that element to itself):

$$
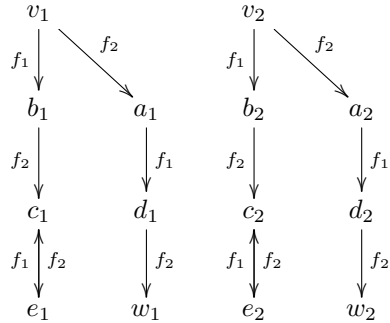
\xymatrix{
v_1\ar[d]_{f_1}\ar[dr]^{f_2}    &     &  v_2\ar[d]_{f_1}\ar[dr]^{f_2} &  & \\
b_1\ar[d]_{f_2} & a_1\ar[d]^{f_1} &  b_2\ar[d]_{f_2} & a_2\ar[d]^{f_1} & \\
c_1\ar[d]_{f_1}& d_1\ar[d]^{f_2}&  c_2\ar[d]_{f_1} & d_2\ar[d]^{f_2} &  \\
e_1\ar[u]_{f_2} & w_1 & e_2\ar[u]_{f_2} & w_2 & 
}$$
\\

Note $f_1(c_i)=e_i$ while $f_2(e_i)=c_i$ for any $i\in \{1,2\}$. There is no $f_2$-arrow coming out of $c_i$'s and similarly, there is no $f_1$-arrow coming out of $e_i$'s (namely they are fixed points of $f_2$ and $f_1$ respectively). Additionally, $|\langle v_i\rangle|=6$, $|\langle b_i\rangle|=|\langle a_i\rangle|=3$, $|\langle c_i\rangle|=|\langle d_i\rangle|=|\langle e_i\rangle|=2$ and $|\langle w_i\rangle|=1$ for any $1\leq i\leq 2$ where ``$\langle \text{ } \rangle$" denotes the substructure generated by those elements inside the bracket.

\begin{proposition}
    The above unary algebra is 1-ultrahomogeneous but not ultrahomogeneous.   
\end{proposition}

\begin{proof}
First, we check that it is not ultrahomogeneous: consider the isomorphism $\Phi$ from $\langle a_1,b_1 \rangle$ onto $\langle a_1, b_2\rangle$ which maps $a_1$ to $a_1$ and $b_1$ to $b_2$. This isomorphism cannot be extended to an automorphism of $\A$. As any automorphism extending $\{(a_1,a_1)\}$ should map $v_1$ to itself while any automorphism extending $\{(b_1,b_2)\}$ should map $v_1$ to $v_2$. 

Second, we show that $\A$ is 1-ultrahomogeneous. For $\langle v_i\rangle$ where $1\leq i\leq 2$, by cardinality reasons, an isomorphism with $\langle v_i\rangle$ as its domain can only have $\langle v_j\rangle$ as its image where $1\leq j\leq 2$. It is easy to see that any such isomorphism can be extended to an automorphism \textemdash simply exchanging elements in $\langle v_i\rangle$ and $\langle v_j\rangle$ correspondingly and fixing the rest (note the isomorphism must map $b_i,a_i,c_i,d_i,e_i,w_i$ to  $b_j,a_j,c_j,d_j,e_j,w_j$ respectively). For $\langle b_i\rangle$ where $1\leq i\leq 2$, by cardinality reasons and the fact that $f_1(b_i)=b_i$ while $f_1(a_i)\not=a_i$, an isomorphism with $\langle b_i\rangle$ as its domain can only have $\langle b_j\rangle$ as its image where $1\leq j\leq 2$ and any such isomorphism can clearly be extended to an automorphism by the same argument for $\langle v_i\rangle$'s. Similarly for $\langle a_i\rangle$'s and $\langle c_i \rangle$'s. For $\langle d_i\rangle$ where $1\leq i\leq 2$, we note that $f_2(d_i)\not= d_i$ and $f_1 (f_2(d_i)) = f_2(d_i)$ while $f_2(c_i)=c_i$ and $f_1(f_2(e_i))\not = f_2(e_i)$, and hence an isomorphism with $\langle d_i\rangle$ as its domain can only have $\langle d_j\rangle$ as its image. For $\langle e_i \rangle$, by cardinality reasons and the cases we've already considered, an isomorphism with $\langle e_i\rangle$ as its domain can only have $\langle e_j\rangle$ as its image. The case for $\langle w_i \rangle$'s is just by cardinality reasons.

\end{proof}

Taking the disjoint union of $\omega$ many copies of $\A$, we have the following as well:

\begin{proposition}
    There exists a uniformly locally finite countably infinite unary algebra which is 1-ultrahomogeneous but not ultrahomogeneous. 
\end{proposition}

To show that 1-homomorphism-homogeneous unary algebras may not be homomorphism-homogeneous, we define a unary algebra $\A=(A,f_1,f_2)$ with two unary functions $f_1,f_2$ as follows (we again omit $f_i$-arrows if $f_i$ maps those elements to themselves):

$$\xymatrix{
v_1\ar[d]_{f_1}\ar[dr]^{f_2}    &     &   &  & \\
b_1 & a_1 &  b_2 & 
}$$
\\

\begin{proposition}
    The above unary algebra is 1-homomorphism-homogeneous but not homomorphism-homogeneous.
\end{proposition}

\begin{proof}
    To see that $\A$ is not homomorphism-homogeneous, we consider the homomorphism from $\{b_1,a_1\}$ to $\{b_1,b_2\}$ which maps  $b_1$ to $b_1$ and $a_1$ to $b_2$. It cannot be extended to an endomorphism as no element can serve as the image of $v_1$. 
    
    We then show that $\A$ is 1-homomorphism-homogeneous. For $\langle v_1\rangle$ (as the domain of the homomorphism), we can always extend the map by fixing $b_2$ regardless of its image. For $\langle b_1\rangle$, by cardinality reasons, the image can only be $\langle b_i\rangle$ or $\langle a_1\rangle$ where $1\leq i\leq 2$. If the image is $\langle b_1\rangle$, we extend it by the identity map. If the image is $\langle a_1\rangle$, we extend it by mapping every point to $a_1$. If the image is $\langle b_2\rangle$, we map every point to $b_2$. Similarly for $\langle a_1\rangle$ as the domain. For $\langle b_2\rangle$, we can just extend the homomorphism by fixing every other point regardless of its image.
\end{proof}

\section{Tree-like unary algebras}

Due to the counterexample in Proposition 2.1, Cor 3.10 in $\cite{q}$ is not true for unary algebras in general as it would imply Theorem (A) for unary algebras. And since Quinn-Gregson used that corollary to prove Theorem (B), even if the second question has a positive answer, the proof strategy should be different. In this section, using a counting argument, we generalize Theorem (B) for a very special kind of unary algebras.

\begin{definition}
    A unary algebra $(A,f_1,...,f_n)$ is connected if the (undirected) graph $(A,R)$ is connected where $Rab:= \bigvee^n_{i=1} (f_i(a)=b\lor f_i(b)=a)$ for any $a\not=b\in A$ . A connected component of a unary algebra is a maximal connected subalgebra.
\end{definition}

When considering $\omega$-categoricity, we can always focus on connected unary algebras because of the following simple observation, which is needed for Corollary 3.7. The proof for the case of mono-unary algebras can be found in \cite[Cor 3.5]{q}.

\begin{proposition}
    A countably infinite unary algebra is $\omega$-categorical if and only if it has finitely many connected components up-to-isomorphism and each connected component is $\omega$-categorical if infinite.
\end{proposition}

\begin{definition}
    An element of a unary algebra $\A=(A,f_1,...,f_n)$ is \textit{cyclic} if $f^{n_1}_{i_1}...f^{n_k}_{i_k}(x)=x$ for some $1 \leq i_1,...,i_k\leq n$ and $n_1,...,n_k\in \mathbb{N}$ where $1\leq n_1+...+n_k$\footnote{Namely, there is some element $m$ (not the identity) in the monoid generated by $f_i^A$s such that $m(x)=x$.}.

    A cyclic element is \textit{properly cyclic} if $f^{n_1}_{i_1}...f^{n_k}_{i_k}(x)=x$ for some $1 \leq i_1,...,i_k\leq n$ where $n_1,...,n_k\in \mathbb{N}$ and one of $f_{i_k}(x),..., f^{n_k}_{i_k}(x),..., f^{n_2}_{i_2}...f^{n_k}_{i_k}(x),f_{i_1}f^{n_2}_{i_2}...f^{n_k}_{i_k}(x)\\,..., f^{n_1}_{i_1}...f^{n_k}_{i_k}(x)$ is different from $x$ (i.e., not every element in the cycle is $x$).     
\end{definition}

For each locally finite unary algebra $\A$, we see immediately that it must have cyclic elements. Moreover, the set of all cyclic elements and the set of all properly cyclic elements are clearly invariant under $Aut(\A)$, which motivates the following definition:

\begin{definition}
    A subset $X\subseteq A$ is a \textit{core} of the unary algebra $\A$ if the following three hold:
\begin{enumerate}
    \item[1)] $X$ is a subalgebra of $\A$.
    \item[2)] $X$ is invariant under $Aut(\A)$.
    \item[3)] $X$ contains all properly cyclic elements of $\A$.
\end{enumerate}
   
\end{definition}

In particular, a core of $A$ is itself a unary algebra. As $X$ is invariant under $Aut(\A)$, $\{g|_X \mid g\in Aut(\A)\}$ is a subgroup of $Aut(X)$, and so it makes sense to talk about orbits of $X$ under $\{g|_X \mid g\in Aut(\A)\}$.

\begin{definition}
    $\A=(A,f_1,...,f_n)$ is a \textit{tree-like} unary algebra if there exists a core $X\subseteq A$ such that the following hold:
\begin{enumerate}
    \item [1)] $A=\bigsqcup_{x\in X} A_x$ (disjoint union as sets) where $A_x=\{x\}\cup \{a\in A\mid f^{n_1}_{i_1}...f^{n_k}_{i_k}(a)= x \text{ for some } 1 \leq i_1,...,i_k\leq n$ and $n_1,...,n_k\in \mathbb{N} \text{ where none of } a, \\f_{i_k}(a),...,f^{n_k}_{i_k}(a),...,f^{n_2}_{i_2}...f^{n_k}_{i_k}(a), f_{i_1}f^{n_2}_{i_2}...f^{n_k}_{i_k}(a), ..., f^{n_1-1}_{i_1}...f^{n_k}_{i_k}(a) \text{ is in } X\}$
    \item[2)] $(A_x,R)$ is a rooted tree with $x$ as its root where $R$ is defined above in Def 3.1.
\end{enumerate}
\end{definition}

Note that for any $a \in A_x\setminus \{x\}$, $f_i(a)\in A_x$ for any $1\leq i\leq n$: otherwise, $a\not\in X$, $f_i(a)\in A_y$ for some $y\not=x$, and thus $a\in A_y$, contradicting the disjointness of $A_x$ and $A_y$. We will refer to this as the closure property of $A_x$'s. And by the definition of cores, we know that $A_x \setminus\{x\}$ has no properly cyclic elements for any $x\in X$. In the following, we may simply call a partial unary algebra $\A$ a rooted tree for convenience if $(A,R)$ is a rooted tree.

Intuitively speaking, a tree-like unary algebra is one with a center where each element of the center has a tree attached. The following is the main result of this section:

\begin{theorem}
 Let $\A= (A,f_1,...,f_n)$ be a tree-like unary algebra with $X\subseteq A$ as a core. If $\A$ is locally finite with finitely many 1-orbits and  $X$ has finitely many $m$-orbits under $\{g|_X\mid g\in Aut(\A)\}$ for any $m\in\mathbb{N}$, then $Aut(\A)$ is oligomorphic. In particular, when $X$ is finite, $Aut(\A)$ is oligomorphic if $\A$ is locally finite with finitely many 1-orbits.
\end{theorem}

\begin{proof}
 Suppose $\A$ is locally finite with finitely many 1-orbits and  $X$ has finitely many $m$-orbits under $\{g|_X\mid g\in Aut(\A)\}$ for any $m\in\mathbb{N}^+$,   we prove that $\A$ is $\omega$-categorical in several steps. 
    For any $x\in X$, we define $ht(a)= \text{min} \{n_1+...+n_k\mid f^{n_1}_{i_1}...f^{n_k}_{i_k}(a)= x \text{ for some } 1 \leq i_1,...,i_k\leq n$ and $n_1,...,n_k\in \mathbb{N}\}$ for any $a\in A_x$\footnote{Note by the definition of a tree-like unary algebra, we know that for any $a\in A$, there is a unique $x\in X$ such that $a\in A_x$. Thus we are justified to write $h(a)$ instead of $h_x(a)$. And $ht(a)$ is exactly the height of $a$ in the rooted tree $(A_x,R)$. }. As $\A$ is locally finite and has finitely many 1-orbits, there exists $m\in \mathbb{N}$ such that $ht(a)\leq m$ for any $a\in A_x$. Thus we can define $ht(A_x)= \text{max} \{ht(a)\mid a\in A_x\}$.
    As $(A_x, R)$ is a tree and there is no properly cyclic element in $A_x\setminus\{x\}$, for any $a\not=b\in A_x\setminus\{x\}$, if $f_i(a)=b$, then $ht(a)=ht(b)+1$. For each $x\in X$, we prove the following claims for $A_x$:

\textbf{Claim 1:} $A_x$ as a partial unary algebra is itself locally finite and has finitely many 1-orbits.

As $\A$ is locally finite, so is $A_x$. For any $a\not=b\in A_x$, suppose $g(a)=b$ for some $g\in Aut(\A)$. By definition, there exist  $1 \leq i_1,...,i_k\leq n$ and $n_1,...,n_k\in \mathbb{N}$ such that $f^{n_1}_{i_1}...f^{n_k}_{i_k}(a)= x $ where $n_1+...+n_k=ht(a)$. If $g(x)\not\in A_x$, then $g(f^{n_1}_{i_1}...f^{n_k}_{i_k}(a))\not\in A_x$. As $g\in Aut(\A)$, we get $f^{n_1}_{i_1}...f^{n_k}_{i_k}(g(a))=f^{n_1}_{i_1}...f^{n_k}_{i_k}(b)\not\in A_x$. As $b\in A_x$, by the closure property of $A_x$, we know that $f_{i_j}^m...f^{n_k}_{i_k}(b)=x$ for some $1\leq j\leq n$ and $m<n_j$. As $n_1+...+n_k=ht(a)$, $f_{i_j}^m...f^{n_k}_{i_k}(a)\in A_x\setminus \{x\}$. In particular, $g(f_{i_j}^m...f^{n_k}_{i_k}(a))=f_{i_j}^m...f^{n_k}_{i_k}(b)\in X$ while $f_{i_j}^m...f^{n_k}_{i_k}(a)\not\in X$. This is impossible as $X$ is invariant under $Aut(\A)$. Therefore, $g(x)\in A_x$. We know  $ht(A_x)=m$ for some $m\in \mathbb{N}$ and $x$ is the unique element in $A_x$ with the property that there exist $a_1,....,a_m\not \in X$ such that $f_{n_1}(a_1)=a_2$,..., $f_{n_{m-1}}(a_{m-1})=a_m$ $f_{n_m}(a_m)=x$ where $1\leq n_1,...,n_m\leq n$. Thus $g(x)=x$. Then it is easy to see that $g(A_x)=A_x$. Thus $g|_{A_x}\in Aut(A_x)$ ($A_x$ as a partial unary algebra). As $\A$ has finitely many 1-orbits, so does $A_x$. 

\textbf{Claim 2:} $A_x$ has finitely many $p$-orbits for any $p\in \mathbb{N}^+$.

We do induction on $ht(A_x)=k$. More precisely, we do induction on the height of locally finite partial unary algebras which are rooted trees with no properly cyclic elements. For $k=0$, $A_x=\{x\}$ and the claim is clearly true. For $k=h+1$, for any $y\in A_x$ with $ht(y)=1$, define $A_y=\{a\in A\mid f^{n_1}_{i_1}...f^{n_k}_{i_k}(a)= y \text{ for some } 1 \leq i_1,...,i_k\leq n$ where $n_1,...,n_k\in \mathbb{N}, \text{ none of } a, f_{i_k}(a),...,f^{n_k}_{i_k}(a),...,\\f^{n_2}_{i_2}...f^{n_k}_{i_k}(a), f_{i_1}f^{n_2}_{i_2}...f^{n_k}_{i_k}(a), ..., f^{n_1-1}_{i_1}...f^{n_k}_{i_k}(a) \text{ is in } X \}$. Note that $y\in A_y \subseteq A_x$ for any $y\in A_x$ with $ht(y)=1$, and $(A_y,R)$ is actually the largest subtree with $y$ as its root. Now for any $a\not=b\in A_y$ where $y\in A_x$ with $ht(y)=1$, if $g(a)=b$ for some $g\in Aut(A_x)$, as $g$ is an automorphism, we have $ht(a)=ht(b)$ and thus $g(y)=y$. Therefore, by definition, we see immediately $g(A_y)=A_y$, and thus $g|_{A_y}\in Aut(A_y)$ ($A_y$ as a partial unary algebra). As $Aut(A_x)$ has finitely many 1-orbits, we thus know that $Aut(A_y)$ has finitely many 1-orbits. As $ht(y)=1$ and $ht(A_x)=h+1$, we have $ht(A_y)\leq h$ and clearly $A_y$ is locally finite. By induction hypothesis, we know that $A_y$ has finitely many $p$-orbits for any $p$. For any $g\in Aut(A_x)$ and any $y\in A_x$ with $ht(y)=1$, we have $ht(g(y))=1$ and $g(A_y)=A_{g(y)}$). Therefore, as $A_x$ has only finitely many 1-orbits, there are only finitely many $A_y$'s with $y\in A_x$ and $ht(y)=1$ up-to-isomorphism. 

Now for any $p\in \mathbb{N}^+$, any $p$-tuple $(a_1,...,a_p)\in A_x\setminus \{x\}$, as $(A_x, R)$ is a rooted tree, there is a unique $p$-tuple $(A_{y_1},...,A_{y_p})$ such that $a_i\in A_{y_i}$ with $ht(y_i)=1$ for any $1\leq i\leq p$. For any $y\in A_x$ with $ht(y)=1$, we define $c(y)=\{1\leq i\leq n\mid f_i(y)=x\}$. Note for any $i\not\in c(y)$, $f_i(y)=y$. For any $y,y'\in A_x$ with $ht(y)=ht(y')=1$, if $A_y \cong A_y'$ and $c(y)=c(y')$, then there exists $g\in Aut(A_x)$ with $g(y)=y'$ and $g(A_y)=A_{y'}$ (taking an isomorphism exchanging $A_y$ and $A_{y'}$ and fixing the rest of $A_x$). Similarly for $(y_1,...y_j), (y'_1,...y'_j)$ with $ht(y_l)=ht(y'_l)=1$ for any $1\leq l\leq j$, if $y_i=y_r$ if and only if $y'_i=y'_r$, $c(y_l)=c(y'_l)$ and $A_{y_l}\cong A_{y'_l}$ for any $1\leq i,r,l\leq n$, then there exists $g\in Aut(A_x)$ with $g(y_l)=y'_l$ and $g(A_{y_l})=A_{y'_l}$ for any $1\leq l\leq j$.

Now we define the following equivalence relation $\sim$ on $p$-tuples of $A_x$ as follows: for any $(a_1,...,a_p),(b_1,...,b_p)\in A_x$, $(a_1,...,a_p)\sim (b_1,...,b_p)$ if and only if the following hold:

\begin{enumerate}
    \item[(1)] $a_i=x$ if and only if $b_i=x$
    \item[(2)] For $a_i\not=a_j\in A_x\setminus\{x\}$ and $b_i\not=b_j\in A_x\setminus\{x\}$, $y_i=y_j$ if and only if $y'_i=y'_j$ where $a_i\in A_{y_i}$, $a_j\in A_{y_j}$, $b_i\in A_{y'_i}$ and $b_j\in A_{y'_j}$ with $ht(y_i)=ht(y_j)=ht(y'_i)=ht(y'_j)=1$
    \item[(3)] For any $a_i\in A_{y_i}$ and $b_i\in A_{y'_i}$  with $ht(y_i)=ht(y'_i)=1$ where $1\leq i\leq n$, say $\Bar{a}$ is the longest subtuple of $(a_1,...,a_p)$ whose elements are all in $A_{y_i}$ and $\Bar{b}$ is the corresponding longest subtuple of $(b_1,...,b_p)$  whose elements are all in $A_{y'_i}$, then  $A_{y_i} \cong A_{y'_i}$, $c(y_i)=c(y'_i)$ and $\Bar{a}$, $\Bar{b}$ are in the same orbit if one identifies $A_{y_i}, A_{y'_i}$ by an isomorphism.    
\end{enumerate}

It is easy to check that if $(a_1,...,a_p)\sim (b_1,...,b_p)$, then there exists $g\in Aut(A_x)$ such that $g(a_i)=b_i$ for any $1\leq i\leq p$, namely $(a_1,...,a_p)$ and $(b_1,...,b_p)$ are in the same orbit. Since there are only finitely many $A_y$'s with $y\in A_x$ and $ht(y)=1$ up-to-isomorphism, and each such $A_y$ has finitely many $i$-orbits for any $i\in \mathbb{N}^+$,  the equivalence relation $\sim$ has finitely many equivalence classes. Thus we can conclude that $A_x$ has finitely many $p$-orbits for any $p\in\mathbb{N}$, and this proves Claim 2.

For any $g\in Aut(\A)$, we have $g|_X\in Aut(X)$ and $g(A_x)\cong A_{g(x)}$ for any $x\in X$ by the definition of $A_x$'s. As $\A$ has only finitely many 1-orbits, there are only finitely many $A_x$'s with $x\in X$ up-to-isomorphism. For any $(a_1,...,a_m)\in A$ where $m\in\mathbb{N}^+$, as $A$ is the disjoint union of $A_x$'s, there is a unique $m$-tuple $(x_1,...,x_m)$ where $x_i\in X$ and $a_i\in A_{x_i}$ for any $1\leq i\leq m$. For convenience, we refer to $(x_1,...,x_m)$ as the index of $(a_1,...,a_m)$. Now we define an equivalence relation on $m$-tuples of $\A$: for any $(a_1,...,a_m), (b_1,...,b_m)\in A$, $(a_1,...,a_m) \sim (b_1,...,b_m)$ holds if and only if the following hold:

\begin{enumerate}
    \item[(1)] The indices of $(a_1,...,a_m)$ and $(b_1,...,b_m)$, say $(x_1,...,x_m)$ and $(x'_1,...,x'_m)$ respectively,  lie in the same $m$-orbit of $X$ under $\{g|_X\mid g\in Aut(\A)\}$. Fix a witness $g\in Aut(\A)$ such that $g(x_1,...,x_m)=(x'_1,...,x'_m)$.
    \item [(2)] For any $1\leq i\leq m$, let $\Bar{a}_i$ be the longest subtuple of $(a_1,...,a_m)$ whose elements are all in $A_{x_i}$ and $\Bar{b}_i$ be the corresponding longest subtuple of $(b_1,...,b_m)$ whose elements are all in $A_{x'_i}$, then $g(\Bar{a}_i)$ and $\Bar{b}_i$ lie in the same orbit of $A_{x'_i}$.
\end{enumerate}

If $(a_1,...,a_m) \sim (b_1,...,b_m)$, then for any $1\leq i\leq m$, there exists $h_i\in Aut(\A)$ such that $h_i(g(\Bar{a}))=\Bar{b}$ where $h_i|_{A_{x'_i}}\in Aut (A_{x'_i})$ and $h_i$ fixes the rest of $A$ pointwise. One can then check that $h_1\circ...\circ h_m\circ g (a_1,...,a_m)=(b_1,...,b_m)$. Therefore, if  $(a_1,...,a_m) \sim (b_1,...,b_m)$, then they are in the same orbit in $\A$. As $X$ has finitely many $m$-orbits under $\{g|_X\mid g\in Aut(\A)\}$ and each $A_x$ has finitely many $m$-orbits, we know that $\sim$ has finitely many equivalence classes. Thus $\A$ has finitely many $m$-orbits for any $m\in\mathbb{N}^+$, and this proves that $Aut(\A)$ is oligomorphic.

\end{proof}

Now by Ryll-Nardzewski theorem and Proposition 3.2, the following result is immediate:

\begin{corollary}
    Let $\A=(A,f_1,...,f_n)$ be a countably infinite unary algebra, if $\A$ is locally finite with finitely many 1-orbits such that each connected component is a tree-like unary algebra with a finite core, then $\A$ is $\omega$-categorical.
\end{corollary}

For mono-unary algebras, we can deduce Theorem (B) as a simple corollary:

\begin{corollary}
If a countably infinite mono-unary algebra $\A$ is locally finite with finitely many 1-orbits, then it is $\omega$-categorical.
\end{corollary}

\begin{proof}
As $\A$ has finitely many 1-orbits, it has finitely many connected components up-to-isomorphism. For each connected component, as $\A$ is locally finite, the set of all cyclic elements forms a nonempty finite subalgebra. Now one can see immediately that each connected component is a tree-like mono-unary algebra with the set of its cyclic elements as the core. Hence by Corollary 3.7, we get that $\A$ is $\omega$-categorical.
\end{proof}

Instead of using the corresponding analogue of \cite[Cor 3.10]{q} for unary algebras, which cannot be true in general as mentioned, we do an explicit counting argument here. This is the key difference between the above proof of Theorem (B) and the one in \cite{q}\footnote{One may wonder if the corresponding analogue of \cite[Cor 3.10]{q} is at least true for tree-like unary algebras. Again the answer is negative and one can see this by the counterexample in Section 2: every connected component is itself a finite core and thus a tree-like unary algebra with a finite core.}.

Besides, it is easy to see that our proof of Claim 2 in Theorem 3.6 also gives us the following simple result about edge-colored directed trees with finite height\footnote{For $\omega$-categorical vertex-colored trees, one can consult \cite{3}.}:

\begin{proposition}
If $\A=(A,R, R_1,...,R_n)$ is an edge-colored directed tree, i.e., $(A,R)$ is a directed tree where $R=\bigsqcup_{1\leq i\leq n} R_i$, then if $\A$ is of finite height with finitely many 1-orbits, then $Aut(\A)$ is oligomorphic.
\end{proposition}


\begin{thebibliography}{30}
\bibitem{I}
A.A. Ivanov, ``Complete theories of unars", \textit{Sib Math J} \textbf{27} (1986), 45-55.

\bibitem{ed}
\'E. Jung\'abel, D. Ma\v{s}ulovi\'c, ``Homomorphism-homogeneous monounary algebras", \textit{Math. Slovaca}
 \textbf{63} (2013), 993–1000.

\bibitem{2}
K. Tent and M. Ziegler, \textit{A Course in Model Theory}, Cambridge University Press, 2012.

\bibitem{3}

R. Barham, ``The $\omega$-categorical trees", \textit{Order} \textbf{34} (2017), 127-138.

\bibitem{q}
T. Quinn-Gregson, ``Ultrahomogeneity and $\omega$-categoricity of mono-unary algebras", \textit{Preprint}, arxiv.org/abs/2603.26616.



\bibitem{S}

Y.E. Shishmarev, ``Categorical theories of a function", \textit{Mathematical Notes of the Academy of Sciences of the USSR} \textbf{11} (1972), 58-63.



\end{thebibliography}
\end{document}